\documentclass[11pt, reqno]{amsart}
\usepackage[T1]{fontenc}
\usepackage{amsfonts}
\usepackage{amsmath}
\usepackage{amsthm}
\usepackage{amssymb}
\usepackage{geometry}
\usepackage{graphicx}
\usepackage{xcolor}
\usepackage{mathtools}
\usepackage[colorlinks=true, linkcolor=blue, citecolor=black]{hyperref}
\usepackage{cleveref}
\usepackage[english]{babel}
\usepackage{lineno}
\usepackage{float}

\newtheorem{theorem}{Theorem}[section]

\newtheorem{lemma}[theorem]{Lemma}

\newtheorem{proposition}[theorem]{Proposition}

\usepackage{tikz}
\usetikzlibrary{positioning}
\usetikzlibrary{decorations,arrows}
\usetikzlibrary{decorations.markings}
\numberwithin{equation}{section}

\usepackage{rustic}
\usepackage[T1]{fontenc}

\emergencystretch=\maxdimen
\title{Strong edge coloring of graphs with maximum degree $6$}
\author{Runze Wang}
\address[]{Department of Mathematics and Computer Science, Augustana College, Rock Island, IL 61201, USA}
\email{runze.w@hotmail.com}
\thanks{}

\begin{document}

\sloppy

\begin{abstract}
    Let $G$ be a graph. Under a strong edge coloring of $G$, every color class is an induced matching. The strong chromatic index of $G$, denoted by $\chi'_s(G)$, is the smallest integer $k$ such that $G$ admits a strong edge coloring with $k$ colors. Denote by $\Delta(G)$ the maximum degree of $G$. In this paper, we prove that every graph $G$ with $\Delta(G)\le 6$ satisfies $\chi'_s(G)\le 57$, improving the best known upper bound $60$.
\end{abstract}

\maketitle

\section{Introduction}

All graphs in this paper are finite and simple. Strong edge coloring was introduced by Fouquet and Jolivet~\cite{FJ}. For a graph $G=(V(G),E(G))$, a \emph{strong edge coloring} of $G$ is an edge coloring such that every color class is an induced matching. Equivalently, it is a proper vertex coloring of $L(G)^2$, the square of the line graph of $G$. The \emph{strong chromatic index} of $G$, denoted by $\chi'_s(G)$, is the smallest integer $k$ such that $G$ admits a strong edge coloring with $k$ colors. When giving $G$ a strong edge coloring, if $G$ is not connected, then we can color each of its components separately. Thus, in this paper, we assume every graph to be connected.

For $v\in V(G)$, denote by $d(v)$ the degree of $v$. Denote by $\Delta(G)$ the maximum degree of $G$. Erd\H{o}s and Ne\v{s}et\v{r}il~\cite{EN} conjectured that
\begin{align*}
\chi'_s(G)\le
\begin{cases}
\frac{5}{4}\Delta^2, & \text{if $\Delta$ is even},\\
\frac{1}{4}(5\Delta^2-2\Delta+1), & \text{if $\Delta$ is odd},
\end{cases}
\end{align*}
where $\Delta=\Delta(G)$. If true, the conjecture upper bound will be sharp, as it can be attained by a balanced blow-up of $C_5$.

For this conjecture, the case $\Delta\le 2$ is trivial. The case $\Delta=3$ was proven by Andersen~\cite{An} and independently by Hor\'ak, He, and Trotter~\cite{HHT}. For $\Delta=4$, Huang, Santana, and Yu~\cite{HSY} proved the upper bound $21$, one more than the conjectured value $20$. For $\Delta=5$, Zang~\cite{Za} proved the upper bound $37$, while the conjectured value is $29$. For other recent progress on strong edge coloring, see e.g.~\cite{CCZZ,CHYZ,LL,LLH,LLY,NY,Wa1,Wa2}.

By an elementary greedy coloring, we have $\chi'_s(G)\le 2\Delta(\Delta-1)+1$. Fouquet and Jolivet~\cite{FJ} proved that if $G$ is not isomorphic to $C_5$ or $K_2$, then
\begin{align*}
\chi'_s(G)\le2\Delta(\Delta-1).
\end{align*}
In particular, for $\Delta=6$, this gives the upper bound $60$.

Our main result lowers this upper bound to $57$.

\begin{theorem}\label{thm:main}
Let $G$ be a graph with $\Delta(G)\le 6$. Then
\begin{align*}
\chi'_s(G)\le 57.
\end{align*}
\end{theorem}

It suffices to prove this theorem for $6$-regular graphs, as for any positive integer $k$ and any graph $H$ with $\Delta(H)\le k$, $H$ is a subgraph of a $k$-regular graph $H'$; and it is clear that $\chi'_s(H)\le \chi'_s(H')$ for any $H\subseteq H'$.

For graphs of girth at least $5$, we can apply the following lemma by Zang~\cite{Za}.

\begin{lemma}[Zang~\cite{Za}, Lemma~3.2]
If $G$ is a $\Delta$-regular graph with $\Delta\ge 3$ and girth at least $5$, then
\begin{align*}
    \chi'_s(G)\le2\Delta^2-3\Delta+2.
\end{align*}
\end{lemma}

So, for a $6$-regular graph $G$ of girth at least $5$,
\begin{align*}
\chi'_s(G)\le2\cdot 6^2-3\cdot 6+2=56.
\end{align*}
Accordingly, it remains only to consider $6$-regular graphs of girth at most $4$. In Section \ref{sec:triangle}, we study graphs containing a triangle. In Section \ref{sec:C4}, we have the main part of the proof that concerns triangle-free graphs containing a $4$-cycle.

\section{A rooted partial coloring and the triangle case}\label{sec:triangle}

In a graph $G$, for two distinct edges $e$ and $e'$, we define the distance between $e$ and $e'$ to be the distance between the corresponding vertices in the line graph $L(G)$. Thus, if $e$ and $e'$ are incident, then they have distance $1$; if $e$ and $e'$ are not incident but are connected by a third edge, then they have distance $2$. We say that $e$ and $e'$ \emph{conflict} if their distance is at most $2$. Thus, two edges conflict if and only if they cannot receive the same color in a strong edge coloring. For $e\in E(G)$, let $N_s(e)$ denote the set of edges in $G$ that conflict with $e$. We call $N_s(e)$ the \emph{strong neighborhood} of $e$.

For a vertex $v\in V(G)$, denote by $G-v$ the subgraph of $G$ obtained by removing $v$ and all edges incident with $v$ from $G$. In the following lemma, we give $G-v$ a rooted greedy strong edge coloring. This lemma is a special case of Lemma 2.1 in \cite{Za}, applied to $6$-regular graphs. We include a proof for completeness.

\begin{lemma}\label{lem:rooted}
Let $G$ be a $6$-regular graph and let $v\in V(G)$. Then there is a strong edge coloring of $G-v$ with $55$ colors.
\end{lemma}

\begin{proof}
For two vertices $v_1,v_2\in V(G)$, denote by $dist_G(v_1,v_2)$ the distance between $v_1$ and $v_2$. For an edge $e=xy$, define
\begin{align*}
\rho(e)=\min\{dist_G(v,x),dist_G(v,y)\}.
\end{align*}
Order the edges in $G-v$ in decreasing order of $\rho(\cdot)$, with ties broken arbitrarily. We will color them greedily in this order. Conflicts are always measured in the original graph $G$.

Let $e=xy$ be an edge that is about to be colored. Choose the notation such that
\begin{align*}
dist_G(v,x)=\rho(e)=r.
\end{align*}
Since $e$ is not incident with $v$, we have $r\ge1$. Let $u$ be the predecessor of $x$ on a shortest $v$--$x$ path. Thus
\begin{align*}
dist_G(v,u)=r-1.
\end{align*}
It is possible that $u=v$. Moreover $u\ne y$, since otherwise $\rho(e)=dist_G(v,y)=r-1$.

Every edge incident with $u$ is still uncolored. Indeed, for $uw\in E(G)$, if it is incident with $v$, which means $u=v$ or $w=v$, then it is deliberately left uncolored, because in this lemma we only color the edges in $G-v$; if it is not incident with $v$, then
\begin{align*}
\rho(uw)\le dist_G(v,u)=r-1<r=\rho(e),
\end{align*}
so $uw$ occurs later in our decreasing-$\rho$ order, and it will be colored later.

We now bound the number of already colored edges that conflict with $e$. Among the edges incident with $x$, both $e$ and $xu$ are uncolored, so there are at most $6-2=4$ colored edges incident with $x$. Among the edges incident with $y$, $e$ is uncolored, so there are at most $6-1=5$ colored edges incident with $y$.

Consider next conflicts at distance $2$. Every edge incident with $u$ is still uncolored. Each of four neighbors of $x$ other than $u$ and $y$ can contribute at most five further edges, and each of the five neighbors of $y$ other than $x$ can contribute at most five further edges. Therefore the number of colored edges that conflict with $e$ is at most
\begin{align*}
4+5+4\cdot 5+5\cdot 5=54.
\end{align*}
Hence there is a feasible color for $e$ as we have $55$ colors. Greedy coloring therefore succeeds for every edge not incident with $v$.
\end{proof}

Now we can solve the case where $G$ has girth $3$.

\begin{proposition}\label{prop:triangle}
Let $G$ be a $6$-regular graph containing a triangle. Then
\begin{align*}
\chi'_s(G)\le57.
\end{align*}
\end{proposition}

\begin{proof}
Let
\begin{align*}
T=abc
\end{align*}
be a triangle, where $a,b,c\in V(G)$. Applying Lemma~\ref{lem:rooted} with root $a$, we obtain a partial strong edge coloring of $G$ using colors $1,2,\ldots,55$ where the six edges incident with $a$ are uncolored. Now erase the colors from every colored edge incident with $b$ or $c$. Thus precisely the edges incident with at least one vertex in $T$ are uncolored.

There are three edges in $T$. Since $G$ is $6$-regular, there are $12$ edges with exactly one endpoint in $V(T)$. First, we use the palette $\{1,2,\ldots,57\}$ to color the $12$ edges with exactly one endpoint in $V(T)$. Consider one such edge, say $e=ax$ with $x\notin V(T)$. There are at most $60=10+2\cdot 25$ edges in the strong neighborhood of $e$: $10$ edges sharing an endpoint with $e$, and at most $25$ further edges through each endpoint. However, in this count the edge $bc$ is encountered through both $ab$ and $ac$. Thus at least one repetition is forced and
\begin{align*}
|N_s(e)|\le59.
\end{align*}
All three triangle edges $ab,ac,bc$ conflict with $e$, and all three of them remain uncolored while the $12$ edges with exactly one endpoint in $V(T)$ are being colored. Therefore, whenever one of these $12$ edges is being colored, there are at most $59-3=56$ colored edges conflicting with it. Hence these $12$ edges can be colored greedily with $57$ colors.

It remains to color the three triangle edges. Consider $ab$. In the standard $60=10+2\cdot 25$ count, $bc$ is counted once among the edges incident with $b$ and once again through the neighbor $c$ of $a$; similarly, $ac$ is counted once among the edges incident with $a$ and once again through the neighbor $c$ of $b$. In addition, each of the four edges in the form $cx$ with $x\notin V(T)$ is counted through both the $a$-side and the $b$-side. Hence at least six repetitions occur, and therefore
\begin{align*}
|N_s(ab)|\le60-6=54.
\end{align*}
The same argument applies to $ac$ and $bc$. Thus the three triangle edges may be colored greedily, completing a strong $57$-edge coloring of $G$.
\end{proof}

\section{The triangle-free four-cycle case}\label{sec:C4}

For the rest of the proof assume that $G$ is triangle-free and contains a $4$-cycle. For a vertex $v\in G$, denote by $N(v)$ the neighborhood of $v$. For $V_1,V_2\subseteq V(G)$, denote by $E(V_1,V_2)$ the set of edges with one endpoint in $V_1$ and the other endpoint in $V_2$. For discrete intervals, let $[k]$ denote $\{1,2,\dots,k\}$.

Choose a vertex $v$ on a $4$-cycle and let
\begin{align*}
D_1=N(v)=\{w_1,w_2,\ldots,w_6\}.
\end{align*}
Assume that $w_1$ and $w_2$ are the two neighbors of $v$ in the $4$-cycle. Since $G$ is triangle-free, $D_1$ is an independent set. Let $D_2$ be the set of vertices at distance exactly $2$ from $v$, and for $i\in[6]$ let
\begin{align*}
e_i=vw_i,
\qquad
W_i=N(w_i)\cap D_2.
\end{align*}
We know that $|W_i|=5$ because every $w_i$ has degree $6$, one of its neighbors is $v$, and it has no neighbor in $D_1$.

Let
\begin{align*}
    B=E(D_1,D_2).
\end{align*}
Then $|B|=30$.

For $i\in[6]$, define
\begin{align}
q_i=|E(D_1\setminus\{w_i\},W_i)|.
\label{eq:qdef}
\end{align}
As $w_1$ and $w_2$ have a common neighbor in $D_2$, which is the fourth vertex other than $v,w_1,w_2$ in the $4$-cycle, we have
\begin{align}
q_1\ge1,
\qquad
q_2\ge1.
\label{eq:qpositive}
\end{align}

\subsection{The exact strong neighborhood of a root edge}

Let $v$ be the root of $G$. Then $e_1,e_2,\dots,e_6$ are the root edges.

\begin{lemma}\label{lem:exact}
For each $i\in[6]$,
\begin{align*}
|N_s(e_i)|=60-q_i.
\end{align*}
\end{lemma}

\begin{proof}
For $x\in D_2$, let
\begin{align*}
r(x)=|N(x)\cap D_1|.
\end{align*}
Every $x\in W_i$ is adjacent to $w_i$, and the other $r(x)-1$ edges between $x$ and $D_1$ are precisely those counted in \eqref{eq:qdef}. Therefore
\begin{align*}
q_i=\sum_{x\in W_i}(r(x)-1),
\end{align*}
and hence
\begin{align}
\sum_{x\in W_i}r(x)=5+q_i.
\label{eq:rsum}
\end{align}

We now partition $N_s(e_i)$ into three classes. 

First, it contains the other five edges incident with $v$. 

Second, it contains all $30$ edges in $B$: an edge in $B$ incident with $w_i$ shares an endpoint with $e_i$, while an edge $w_jx\in B$ with $j\ne i$ is at distance $2$ from $e_i$ through the root edge $e_j=vw_j$. 

Third, the remaining edges conflicting with $e_i$ are exactly the edges outside $B$ incident with a vertex in $W_i$. Indeed, any edge conflicting with $e_i$ must be incident with $v$ or $w_i$, or have an endpoint adjacent to $v$ or $w_i$. Because $D_1$ is independent, the edges arising through a neighbor of $v$ are precisely the root edges or the edges in $B$; the only additional edges arising through a neighbor of $w_i$ are the edges incident with $W_i$.

For $x\in W_i$, exactly $r(x)$ of the six edges incident with $x$ belong to $B$, so $6-r(x)$ edges incident with $x$ lie outside $B$. Moreover, no edge has both endpoints in $W_i$, because if $x,y\in W_i$ are adjacent, then $w_i x y w_i$ will be a triangle, a contradiction. Thus, these outside-$B$ edges are counted exactly once, and their number is
\begin{align*}
\sum_{x\in W_i}(6-r(x))
=30-(5+q_i)
=25-q_i
\end{align*}
by \eqref{eq:rsum}. The three classes are pairwise disjoint. Therefore
\begin{align*}
|N_s(e_i)|=5+30+(25-q_i)=60-q_i.
\end{align*}
\end{proof}

\subsection{An erasing-and-recoloring argument and Hall's theorem}

Applying Lemma~\ref{lem:rooted} with root $v$, we can color every edge except $e_1,e_2,\ldots,e_6$ with $55$ colors $1,2,\ldots,55$. We call these $55$ colors \emph{old colors}. Let
\begin{align*}
\alpha=56,
\qquad
\beta=57
\end{align*}
be two \emph{new colors}.

We will occasionally erase the old colors from a set of edges in $B$ and recolor those edges with one or both of the new colors. The following estimate gives a lower bound on the number of colors available for each root edge $e_i$ after the recoloring.

\begin{lemma}\label{lem:erasing}
Let $S\subseteq B$ be a set of $t$ edges. Assume that the old colors on the edges in $S$ are erased and these $t$ edges are recolored using $k=1\text{ or }2$ new colors, where each new color class is an induced matching. For $i\in[6]$, let $L_i$ be the set of colors in $\{1,2,\ldots,57\}$ that are available for $e_i$ after this recoloring. Then
\begin{align*}
|L_i|\ge2+q_i+t-k.
\end{align*}
\end{lemma}

\begin{proof}
By Lemma~\ref{lem:exact}, exactly $60-q_i$ edges conflict with $e_i$. These include the other five root edges, which remain uncolored, and all $t$ edges in $S$. Thus at most
\begin{align*}
(60-q_i)-5-t
\end{align*}
edges conflicting with $e_i$ retain old colors. They forbid at most $(60-q_i)-5-t$ old colors to be used on $e_i$. The recolored edges in $S$ use $k$ new colors, so in total at most
\begin{align*}
(60-q_i)-5-t+k
\end{align*}
colors are forbidden for $e_i$. Thus
\begin{align*}
|L_i|
\ge57-\bigl((60-q_i)-5-t+k\bigr)
=2+q_i+t-k.
\end{align*}
\end{proof}

We will use the following consequence of Hall's theorem \cite{Ha}.

\begin{lemma}\label{lem:Hall}
Let $L_1,L_2,\ldots,L_6$ be six subsets of $\{1,2,\dots,57\}$ and let $q_1,q_2,\ldots,q_6$ be nonnegative integers satisfying
\begin{align*}
|L_i|\ge c+q_i
\end{align*}
for $i\in[6]$ and some integer $c$. Let
\begin{align*}
q_{(1)}\le q_{(2)}\le\cdots\le q_{(6)}
\end{align*}
be the increasing rearrangement of $q_1,q_2,\ldots,q_6$. If
\begin{align}
c+q_{(j)}\ge j
\label{eq:HallCriterion}
\end{align}
for each $j\in[6]$, then $L_1,L_2,\ldots,L_6$ have a system of distinct representatives.
\end{lemma}

\begin{proof}
For $k\in[6]$, we arbitrarily choose $k$ sets from $L_1,L_2,\ldots,L_6$, and let $S$ be the union of these $k$ sets. Among the corresponding $k$ values of $q_i$, at least one of them is at least $q_{(k)}$. Hence 
\begin{align*}
|S|\ge c+q_{(k)}\ge k.
\end{align*}
So, by Hall's theorem, we have a system of distinct representatives for $L_1,L_2,\ldots,L_6$.
\end{proof}

\subsection{Private vertices and an auxiliary graph}

A vertex $x\in W_i$ is called \emph{private} if
\begin{align*}
N(x)\cap D_1=\{w_i\}.
\end{align*}
Thus if $x$ is private, then $r(x)=|N(x)\cap D_1|=1$. Let $P_i$ be the set of private vertices in $W_i$. We call $P_i$ the $i$-th \emph{private vertex class}. An edge $w_ix\in B$ is called a \emph{private edge} if $x\in P_i$, and we let $\{w_ix:x\in P_i\}$ be the $i$-th \emph{private edge class}. The sets $P_1,P_2,\dots,P_6$ are pairwise disjoint. Since every nonprivate vertex in $W_i$ contributes at least $1$ to $q_i$, we have
\begin{equation}\label{eq:private}
|P_i|\ge5-q_i.
\end{equation}

Define an auxiliary graph $K$ with vertex set $B$, where two vertices in $K$ are adjacent if and only if the corresponding edges in $B$ do \emph{not} conflict in $G$. Thus a clique in $K$ corresponds exactly to an induced matching in $G$ consisting of edges in $B$.

\begin{lemma}\label{lem:privateconflict}
Let $x\in P_i$ and $y\in P_j$ with $i\ne j$. Then the two vertices in $K$ corresponding to $w_ix$ and $w_jy$ are adjacent if and only if $xy\notin E(G)$.
\end{lemma}

\begin{proof}
The edges $w_ix$ and $w_jy$ are disjoint because the private sets are pairwise disjoint. Since $D_1$ is independent,
\begin{align*}
w_iw_j\notin E(G).
\end{align*}
As $x\in P_i$ and $y\in P_j$, we have
\begin{align*}
xw_j\notin E(G),
\qquad
yw_i\notin E(G).
\end{align*}
Hence the only possible edge joining an endpoint of $w_ix$ to an endpoint of $w_jy$ is $xy$. Therefore the two private edges do not conflict if and only if $xy\notin E(G)$.
\end{proof}

\subsection{The sparse-overlap case}

We first handle the case where every $q_i$ is at most $1$.

\begin{lemma}\label{lem:sparse}
If
\begin{align*}
\max_{i\in[6]}q_i\le1,
\end{align*}
then $G$ has a strong $57$-edge coloring.
\end{lemma}

\begin{proof}
By \eqref{eq:private}, for any $i\in[6]$,
\begin{equation}\label{eq:Pi4}
|P_i|\ge4
\end{equation}
We first show that $K$ contains a triangle whose three vertices correspond to three private edges.

Choose $x\in P_1$. For $j=2,3,\ldots,6$, let
\begin{align*}
A_j=\{y\in P_j:xy\notin E(G)\}.
\end{align*}
The five sets $P_2,P_3,\ldots,P_6$ contain at least $20$ vertices in total. One of the six neighbors of $x$ is $w_1$, so $x$ has at most five neighbors in $\bigcup_{j=2}^6 P_j$. Hence
\begin{equation}\label{eq:Asum}
\sum_{j=2}^6|A_j|\ge15.
\end{equation}
Suppose that $K$ does not contain a triangle whose three vertices correspond to three private edges. If $j\ne k$, $y\in A_j$, and $z\in A_k$, then $xy,xz\notin E(G)$. If we also have $yz\notin E(G)$, then Lemma~\ref{lem:privateconflict} implies that the three vertices in $K$ corresponding to the private edges $xw_1$, $yw_j$, and $zw_k$ are pairwise adjacent, and hence form a triangle, a contradiction. Thus any two nonempty $A_j,A_k$ with $j\ne k$ are completely joined in $G$.

Choose $j$ with $A_j\ne\varnothing$ and let $y\in A_j$. Since $A_j\subseteq P_j\subseteq W_j$ and $|W_j|=5$, we have $|A_j|\le5$. Hence \eqref{eq:Asum} gives
\begin{align*}
\sum_{k\ne j}|A_k|\ge10.
\end{align*}
But $y$ is adjacent to every vertex counted in this sum, contradicting $d(y)=6$. Therefore $K$ contains a triangle $T$ whose three vertices correspond to three private edges in $G$.

These three private edges belong to three distinct private edge classes, since two private edges having the same endpoint in $D_1$ conflict. Let these three private edges be
\begin{align*}
w_{i_1}x_1,\qquad w_{i_2}x_2,\qquad w_{i_3}x_3,
\end{align*}
where $x_\ell\in P_{i_\ell}$ for $\ell\in[3]$. In particular, $i_1,i_2,i_3$ are distinct. For each $i\in[6]$, let
\begin{align*}
\widetilde P_i=P_i\setminus\{x_1,x_2,x_3\}.
\end{align*}
Since at most one vertex is removed from each $P_i$, \eqref{eq:Pi4} gives
\begin{align*}
|\widetilde P_i|\ge3
\end{align*}
for every $i\in[6]$.

Suppose that there is no edge in $K$ joining two vertices corresponding to private edges $w_ix$ and $w_jy$ with
\[
x\in\widetilde P_i,\qquad y\in\widetilde P_j,\qquad i\ne j.
\]
Then for every such pair of $x$ and $y$, the vertices in $K$ corresponding to $w_ix$ and $w_jy$ are nonadjacent. By Lemma~\ref{lem:privateconflict}, $xy\in E(G)$. Hence, for any $i\in[6]$ and any $x\in\widetilde P_i$, the vertex $x$ is adjacent to every vertex in $\widetilde P_j$ for each $j\ne i$. Therefore
\begin{align*}
d(x)\ge\sum_{j\ne i}|\widetilde P_j|
\ge5\cdot3=15,
\end{align*}
contradicting $d(x)=6$.

Thus there is an edge $e$ in $K$ whose endpoints correspond to private edges $w_ix$ and $w_jy$ for some
\[
x\in\widetilde P_i,\qquad y\in\widetilde P_j.
\]
Since $x,y\notin\{x_1,x_2,x_3\}$, neither endpoint of $e$ is a vertex in $T$. 

Recolor with $\alpha$ the three private edges corresponding to the vertices in $T$, and recolor with $\beta$ the two private edges corresponding to the endpoints of $e$. Each new color class is an induced matching. Here, for Lemma~\ref{lem:erasing}, we have
\begin{align*}
t=5,
\qquad
k=2.
\end{align*}
So Lemma~\ref{lem:erasing} gives
\begin{align*}
|L_i|\ge5+q_i.
\end{align*}
Now we apply Lemma~\ref{lem:Hall}. For $j\le5$ we have $5+q_{(j)}\ge j$. Moreover, \eqref{eq:qpositive} gives $q_{(6)}\ge1$, so $5+q_{(6)}\ge6$. Therefore Lemma~\ref{lem:Hall} provides distinct colors for the six root edges
$e_i=vw_i$ with $i\in[6]$. This completes the coloring.
\end{proof}

Henceforth we may assume
\begin{equation}\label{eq:qmax2}
\max_{i\in[6]}q_i\ge2.
\end{equation}

\subsection{The remaining possibilities for the auxiliary graph}

\begin{lemma}\label{lem:Ktriangle}
If $K$ contains a triangle, then $G$ has a strong $57$-edge coloring.
\end{lemma}

\begin{proof}
The three $B$-edges corresponding to the three triangle vertices form an induced matching in $G$. Recolor them all with $\alpha$. Then for Lemma~\ref{lem:erasing} we have $t=3$ and $k=1$, so
\begin{align*}
|L_i|\ge4+q_i.
\end{align*}
By \eqref{eq:qpositive}, at least two $q_i$ are positive, hence
\begin{align*}
q_{(5)}\ge1.
\end{align*}
By \eqref{eq:qmax2},
\begin{align*}
q_{(6)}\ge2.
\end{align*}
Therefore in Lemma~\ref{lem:Hall} we have
\begin{align*}
4+q_{(j)}\ge j
\end{align*}
for $j\in[6]$, so we can color the six root edges $e_i$ with $i\in[6]$.
\end{proof}

\begin{lemma}\label{lem:Kmatching}
If $K$ contains a matching of size $2$, then $G$ has a strong $57$-edge coloring.
\end{lemma}

\begin{proof}
This matching of size $2$ represents two disjoint pairs of nonconflicting $B$-edges in $G$. Recolor one pair with $\alpha$ and the other pair with $\beta$. Then we have $t=4$ and $k=2$ for Lemma~\ref{lem:erasing}, so once again
\begin{align*}
|L_i|\ge4+q_i.
\end{align*}
The same inequalities $q_{(5)}\ge1$ and $q_{(6)}\ge2$ used in Lemma~\ref{lem:Ktriangle} verify the hypothesis of Lemma~\ref{lem:Hall}. So by Lemma~\ref{lem:Hall}, we can color the six root edges $e_i$ with $i\in[6]$.
\end{proof}

Therefore, we may assume that $K$ is triangle-free and any matching in $K$ has only one edge.

\begin{lemma}\label{lem:starcase}
If $K$ has at least one edge, then $G$ has a strong $57$-edge coloring.
\end{lemma}

\begin{proof}
If $K$ has at least one edge, then all its edges have a common endpoint, which means $K$ consists of a star and possibly some isolated vertices. Let $c$ be the center of this star. Choose an arbitrary edge in $K$. Its two endpoints correspond to two nonconflicting $B$-edges in $G$. Recolor these two edges with $\alpha$. Thus $t=2$, $k=1$, and Lemma~\ref{lem:erasing} gives
\begin{equation}\label{eq:starlist}
|L_i|\ge3+q_i.
\end{equation}
Let $e_c$ be the edge in $B$ corresponding to the center $c$. For each $i\in[6]$, define
\begin{align*}
P_i'=\{x\in P_i:w_ix\ne e_c\}.
\end{align*}
By \eqref{eq:private}, we have
\begin{equation}\label{eq:Pprime}
|P_i'|\ge|P_i|-1\ge4-q_i.
\end{equation}

Let $x\in P_i'$ and $y\in P_j'$ with $i\ne j$. In $K$, the vertices corresponding to $w_ix$ and $w_jy$ are both different from the center $c$. Since every edge in $K$ is incident with $c$, these two vertices are nonadjacent in $K$, and hence $w_ix$ and $w_jy$ conflict. Lemma~\ref{lem:privateconflict} yields $xy\in E(G)$. Thus if $P_i'$ and $P_j'$ are nonempty, then they are completely joined in $G$.

We now derive the lower bounds on the ordered $q_i$ needed for Lemma~\ref{lem:Hall}. There can be at most three indices with $q_i=0$, because otherwise four corresponding sets $P_i'$ would each have size at least $4$ by \eqref{eq:Pprime}, and a vertex $v$ in one of them would have at least $3\cdot4=12$ neighbors in the other three, contradicting $d(v)=6$. Thus
\begin{equation}\label{eq:starq4}
q_{(4)}\ge1.
\end{equation}
Similarly, there can be at most four indices with $q_i\le 1$, because otherwise five corresponding sets $P_i'$ would each have size at least $3$, and a vertex $v$ in one of them would have at least $4\cdot3=12$ neighbors in the other four, contradicting $d(v)=6$. Thus
\begin{equation}\label{eq:starq5}
q_{(5)}\ge2.
\end{equation}
Finally, there can be at most five indices with $q_i\le 2$, because otherwise all six $P_i'$ would each have size at least $2$, and a vertex $v$ in one of them would have at least $5\cdot2=10$ neighbors in the other five, contradicting $d(v)=6$. Thus
\begin{equation}\label{eq:starq6}
q_{(6)}\ge3.
\end{equation}
By \eqref{eq:starq4}--\eqref{eq:starq6}, for each $j\in[6]$,
\begin{align*}
3+q_{(j)}\ge j,
\end{align*}
which, together with \eqref{eq:starlist}, allows us to apply Lemma~\ref{lem:Hall} to complete the coloring.
\end{proof}

\begin{lemma}\label{lem:emptycase}
If $K$ is edgeless, then $G$ has a strong $57$-edge coloring.
\end{lemma}

\begin{proof}

No recoloring is needed. For each $i\in[6]$, let $L_i$ be the set of colors in $\{1,2,\ldots,57\}$ available for $e_i$. By Lemma~\ref{lem:exact}, exactly $60-q_i$ edges conflict with $e_i$, and five of them are the other five root edges that are uncolored. Therefore
\begin{equation}\label{eq:emptylist}
    |L_i|\ge 57-\bigl((60-q_i)-5\bigr)=2+q_i.
\end{equation}

Since $K$ is edgeless, any two private edges belonging to distinct private edge classes conflict. Lemma~\ref{lem:privateconflict} therefore implies that any two nonempty private vertex classes $P_i$ and $P_j$ are completely joined in $G$.

If three $q_i$ are $0$, then the three corresponding private vertex classes each have size $5$ by \eqref{eq:private}, and a vertex $v$ in one of them would have at least $2\cdot 5=10$ neighbors in the other two, contradicting $d(v)=6$. Thus
\begin{equation}\label{eq:emptyq3}
q_{(3)}\ge1.
\end{equation}
If four $q_i$ are at most $1$, then the four corresponding private vertex classes each have size at least $4$, and a vertex $v$ in one of them would have at least $3\cdot 4=12$ neighbors in the other three, contradicting $d(v)=6$. Thus
\begin{equation}\label{eq:emptyq4}
q_{(4)}\ge2.
\end{equation}
If five $q_i$ are at most $2$, then the five corresponding private vertex classes each have size at least $3$, and a vertex $v$ in one of them would have at least $4\cdot 3=12$ neighbors in the other four, contradicting $d(v)=6$. Thus
\begin{equation}\label{eq:emptyq5}
q_{(5)}\ge3.
\end{equation}
Finally, if all six $q_i$ are at most $3$, then all six private vertex classes have size at least $2$, and a vertex $v$ in any one class would have at least $5\cdot 2=10$ neighbors in the other five, contradicting $d(v)=6$. Thus
\begin{equation}\label{eq:emptyq6}
q_{(6)}\ge4.
\end{equation}
By \eqref{eq:emptyq3}--\eqref{eq:emptyq6}, for each $j\in[6]$,
\begin{align*}
2+q_{(j)}\ge j,
\end{align*}
which, together with \eqref{eq:emptylist}, allows us to apply Lemma~\ref{lem:Hall} to complete the coloring.
\end{proof}

The preceding lemmas exhaust all possibilities for $K$, so we have proven the following proposition.

\begin{proposition}\label{prop:C4}
If $G$ is $6$-regular, triangle-free, and contains a $4$-cycle, then
\begin{align*}
\chi'_s(G)\le57.
\end{align*}
\end{proposition}

\section{Proof of the main theorem}
To prove Theorem~\ref{thm:main}, as mentioned earlier, it suffices to prove it for $6$-regular graphs of girth at most $4$. Thus, combining Proposition~\ref{prop:triangle} and Proposition~\ref{prop:C4}, we have proven Theorem~\ref{thm:main}.

\section*{Declaration of generative AI use}
The author used OpenAI's GPT-5.6 Sol and GPT-6 Astra models during the preparation of this manuscript for assistance with mathematical exploration and verification, as well as with drafting and language editing. The author independently reviewed, verified, and edited the content and takes full responsibility for the final manuscript.

\end{document}